\documentclass[11pt]{amsart}
\usepackage{amsmath}
\usepackage{tikz}
\usetikzlibrary{matrix,arrows}
\usepackage{amsfonts,amssymb,amsthm, txfonts, pxfonts,amscd}
\usepackage[stretch=10]{microtype}
\usepackage{hyperref}
\def\struckint{\mathop{%
\def\mathpalette##1##2{\mathchoice{##1\displaystyle##2}%
 {##1\textstyle##2}{##1\scriptstyle##2}{##1\scriptscriptstyle##2}}%
\mathpalette
{\vbox\bgroup\baselineskip0pt\lineskiplimit-1000pt\lineskip-1000pt
\halign\bgroup\hfill$}
{##$\hfill\cr{\intop}\cr\diagup\cr\egroup\egroup}%
}\limits}
\usepackage{color}
\usepackage{booktabs,caption,fixltx2e}
\usepackage[flushleft]{threeparttable}

\usepackage{amsmath}
\usepackage{tikz}
\usetikzlibrary{matrix,arrows}
\usepackage{amsfonts,amssymb,amsthm, txfonts, pxfonts,amscd}

\usepackage{soul,color}

\def\Nb{\mathop{\mathbb{N}_{}}\nolimits}

\def\fin{\mathop{\textsc{Fin}}\nolimits}

\def\equalinlaw{=_{\mathcal{D}}}

\def\finite{\mathop{\subset_f}}

\def\equalinlaw{\mathop{=_{\mathcal{D}}}\nolimits}			
\def\Xbf{\mathop{\mathbf{X}_{}}\nolimits}		
\def\xbf{\mathop{\mathbf{x}_{}}\nolimits}	
\def\ybf{\mathop{\mathbf{y}_{}}\nolimits}				
\def\Ybf{\mathop{\mathbf{Y}_{}}\nolimits}					
\def\Zbf{\mathop{\mathbf{Z}_{}}\nolimits}

\def\Nb{\mathop{\mathbb{N}_{}}\nolimits}					
\def\ar{\mathop{\alpha}\nolimits}	
\def\dom{\mathop{\mathrm{dom}}\nolimits}

\def\Zb{\mathop{\mathbb{Z}_{}}\nolimits}

\def\FR{\mathop{\mathcal{F}_{\mathcal{R}}}\nolimits}

\def\RNcong{\mathop{\mathcal{R}_{\mathbb{N}}^{\cong}}\nolimits}
\def\Rncong{\mathop{\mathcal{R}_{[n]}^{\cong}}\nolimits}
\def\Rn{\mathop{\mathbf{R}_n}\nolimits}
\def\Ro{\mathop{\mathcal{R}_{[0,1]}}\nolimits}

\newtheorem{thm}{Theorem}[section]

\newtheorem{prop}[thm]{Proposition}

\newtheorem{defn}[thm]{Definition}
\newtheorem{example}[thm]{Example}

\newtheorem{rmk}[thm]{Remark}

\makeatletter

\def\dotminussym#1#2{%
  \setbox0=\hbox{$\m@th#1-$}%
  \kern.5\wd0%
  \hbox to 0pt{\hss\hbox{$\m@th#1-$}\hss}%
  \raise.6\ht0\hbox to 0pt{\hss$\m@th#1.$\hss}%
  \kern.5\wd0}

\mathchardef\mhyphen="2D

\begin{document}

\title{Relational exchangeability}
\author{Harry Crane and Walter Dempsey}
\address {Department of Statistics \& Biostatistics, Rutgers University, 110 Frelinghuysen Avenue, Piscataway, NJ 08854, USA}
\email{hcrane@stat.rutgers.edu}
\urladdr{\url{http://stat.rutgers.edu/home/hcrane}}
\address {Department of Statistics, University of Michigan, 1085 S. University Ave,  Ann Arbor, MI 48109, USA}
\email{wdem@umich.edu}
\thanks{H.\ Crane is partially supported by NSF grants CNS-1523785 and CAREER DMS-1554092.}

\date{\today}

\maketitle

\begin{abstract}

A {\em relationally exchangeable structure} is a random combinatorial structure whose law is invariant with respect to relabeling its relations, as opposed to its elements.
Aside from exchangeable random partitions, examples include edge exchangeable random graphs and hypergraphs, path exchangeable processes, and a range of other network-like structures that arise in statistical applications.
We prove a de Finetti-type structure theorem for the general class of relationally exchangeable structures.
\end{abstract}

\section{Introduction}\label{section:introduction}
Consider a random sequence $\mathbf{X}=(X_1,\ldots,X_n)$ taking values in an at most countable set ${S}$.
For concreteness, we may regard ${S}$ as a set of species, so that $\Xbf$ records the species for a random sample of animals from a certain population.
From $\Xbf$, we define an equivalence relation $\sim_{\Xbf}$ on $[n]:=\{1,\ldots,n\}$ by
\begin{equation}\label{eq:equiv-relation}
i\sim_{\Xbf}j\quad\text{if and only if}\quad X_i=X_j.
\end{equation}
We write $\Pi(\Xbf)=\{B_1,B_2,\ldots\}$ to denote the set partition whose blocks $B_1,B_2,\ldots$ are the equivalence classes induced by $\sim_{\Xbf}$.  
We sometimes write $\sim_{\Pi}$ in place of $\sim_{\Xbf}$ when convenient.

Suppose now that $\Xbf=(X_1,\ldots,X_n)$ is {\em exchangeable}, meaning that 
\[(X_{\sigma(1)},\ldots,X_{\sigma(n)})\equalinlaw(X_1,\ldots,X_n)\]
for all permutations $\sigma:[n]\to[n]$, where $\equalinlaw$ denotes {\em equality in law}.
Exchangeability of $\Xbf$ induces exchangeability on the equivalence relation $\Pi:=\Pi(\Xbf)$ as in \eqref{eq:equiv-relation} in the sense that $\Pi^{\sigma}\equalinlaw\Pi$ for all permutations $\sigma:[n]\to[n]$, with $\Pi^{\sigma}$ defined by
\[i\sim_{\Pi^{\sigma}} j\quad\text{if and only if}\quad\sigma^{-1}(i)\sim_{\Pi}\sigma^{-1}(j).\]
Below we present a theory of general {\em relationally exchangeable structures} which contains the above example as a special case and generalizes Kingman's correspondence for exchangeable random partitions \cite{Kingman1978a}.

To appreciate the generality in which we work, let $S$ again be a countable set but now consider a sequence $\Xbf=(X_1,\ldots,X_n)$ taking values in $S\times S$, so that each $X_i$ is an ordered pair $(c_i,r_i)$.
For a concrete application, we can think of each $X_i=(c_i,r_i)$ as identifying the caller and receiver of a phone call sampled uniformly at random from a telephone call log.
While the data comes in the form of a sequence $\Xbf$, the structure relating the calls can be represented by a network-like object as in Figure \ref{fig:calls}.
Since the sequence $\Xbf$ is exchangeable, the model assigns equal probability to structures that are isomorphic up to relabeling edges as in Figures \ref{fig:calls}(a) and \ref{fig:calls}(b).

\begin{figure}[!t]
\includegraphics[scale=0.5]{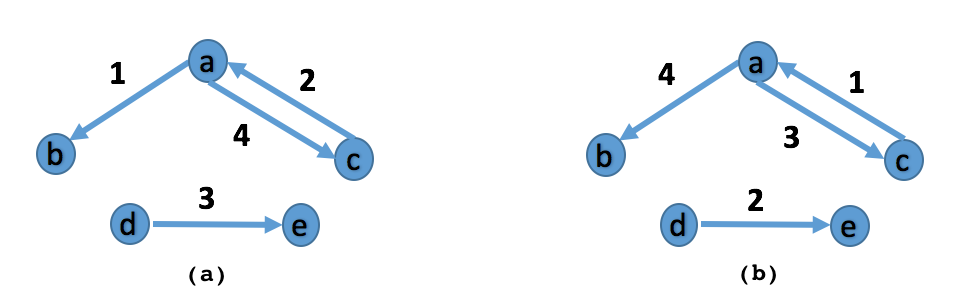}
\caption{(a) Representation of sequence of ordered pairs in $S=\{a,b,c,d,e\}$ with $X_1=(a,b), X_2=(c,a), X_3=(d,e), X_4=(a,c)$.  (b) Representation of sequence of ordered pairs in $S=\{a,b,c,d,e\}$ with $X'_1=(c,a), X'_2=(d,e), X'_3=(a,c), X'_4=(a,b)$.  The object in (b) can be obtained from that in (a) by reordering $(X_1,X_2,X_3,X_4)$ as $(X_2,X_3,X_4,X_1)$.  In both (a) and (b), the edges are labeled according to where they appear in sequence.}
\label{fig:calls}
\end{figure}

Just as we disregarded the species names in passing from $\Xbf$ to its induced equivalence relation $\sim_{\Xbf}$ in \eqref{eq:equiv-relation}, we may also disregard the vertex labels in Figure \ref{fig:calls} to obtain an {\em edge-labeled graph}, as shown in Figure \ref{fig:calls-exch}.
The edge-labeled graph is defined formally as the equivalence class $\Xbf_{\cong}$ of all sequences $\Xbf'$ that yield the same structure as $\Xbf$ after removing vertex labels:
\begin{equation}\label{eq:edge-labeled}
\Xbf_{\cong}:=\{\Xbf':[n]\to S\times{S}: \rho \Xbf'=\Xbf\text{ for some bijection }\rho:S\to S\},
\end{equation} 
where here we overload notation by allowing the bijection $\rho:S\to S$ to act on $S\times S$ by $(c,r)\mapsto(\rho(c),\rho(r))$, so that $\rho\Xbf':=\rho\circ\Xbf':[n]\to S\times{S}$ is well-defined by composition of functions.
Exchangeability of $\Xbf$ immediately implies {\em edge exchangeability} of the edge-labeled graph associated to $\Xbf_{\cong}$ \cite{CD2016e2}.

\begin{figure}[!t]
\includegraphics[scale=0.5]{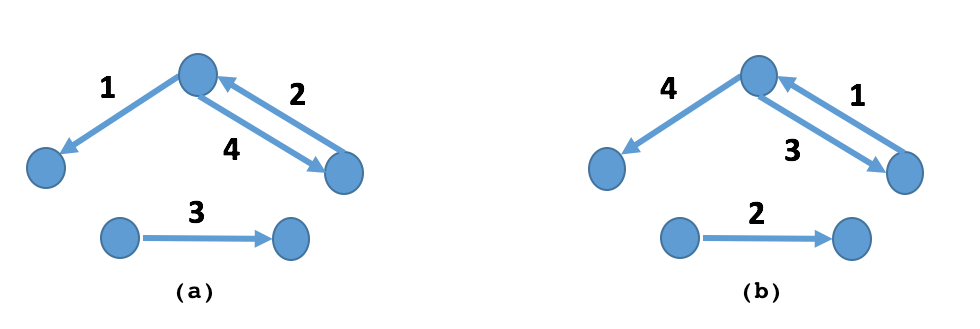}
\caption{(a) Edge-labeled graph induced by disregarding vertex labels in Figure \ref{fig:calls}(a).  (b) Edge-labeled graph induced by disregarding vertex labels in Figure \ref{fig:calls}(b).  The graphs in (a) and (b) have equal probability under an edge exchangeable model.}
\label{fig:calls-exch}
\end{figure}

For another example, let $S$ represent a set of scientists and let
$\Xbf=(X_1,\ldots,X_n)$ record names of coauthors on a collection of
$n$ scientific journal articles sampled uniformly from a database 
(e.g., arXiv).
Each $X_i$ is a finite subset $\{s_{i,1},\ldots,s_{i,k_i}\}\subset S$.
The assumption of uniform sampling again makes $\Xbf$ an exchangeable
sequence, which in turn induces exchangeability on the associated
edge-labeled hypergraph constructed by disregarding vertex labels in
the induced hypergraph structure in a manner analogous to the
construction of edge-labeled graphs from the equivalence class
\eqref{eq:edge-labeled}.

For another example, let $S$ be a set of Internet Protocol (IP) addresses and let each entry of $\Xbf=(X_1,\ldots,X_n)$ correspond to the path taken by a message sent between two IP addresses over the Internet.  Each $X_i$ corresponds to a path $(s_i,a_{i,1},\ldots,a_{i,k_i},t_i)$ from source $s_i$ to target $t_i$ by passing through the intermediate nodes $a_{i,1},\ldots,a_{i,k_i}$, altogether indicating that the path traversed $s$ to $a_{i,1}$ to $a_{i,2}$ and so on until passing from $s_{i,k_1}$ to $t_i$.
If $\Xbf$ was obtained, for example, by uniform random sampling of source-target pairs $(s_i,t_i)$ and then by  applying an algorithm, such as traceroute, to obtain a path between $s_i$ and $t_i$, then the sequence of paths $\Xbf$ is exchangeable and, therefore, so is the induced path-labeled structure shown in Figure \ref{fig:path}.

Another example in the realm of networks is to take each $X_i$ to be an $r$-step ego network, obtained, e.g., by snowball sampling a neighborhood of size $r$ from a randomly chosen vertex.  
The relationally labeled structure is constructed by piecing together the
neighborhoods obtained from the $r$-neighborhoods of $n$ randomly
chosen ``egos'' in this population.  We omit the details of this
example and instead move on to our general treatment.

Myriad other data structures arise according to a similar recipe: let $S$ be a set of elements and $\mathcal{R}$ be a set of relations on the elements in $S$; sample an exchangeable sequence $\Xbf$ taking values in $\mathcal{R}$; construct a structure from $\Xbf$ (as in Figure \ref{fig:calls}) and obtain the corresponding relationally labeled structure by removing the vertex labels (as in Figure \ref{fig:calls-exch}).
The resulting structure is {exchangeable} with respect to relabeling of its relations, a property which we call {\em relational exchangeability}.
We now define this notion formally and prove a generic structure theorem for infinite relationally exchangeable structures.  The formal statement is given in Theorem \ref{thm:general}.

\begin{figure}[!t]
\includegraphics[scale=0.5]{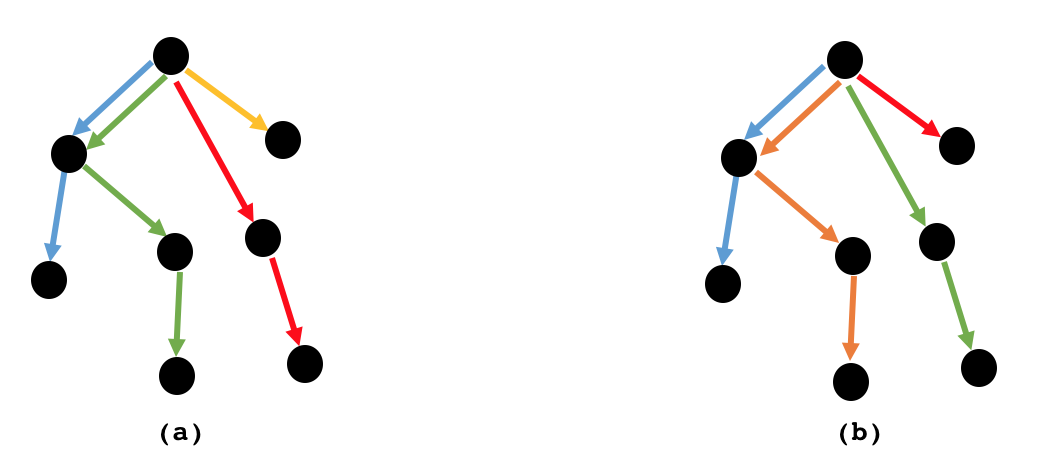}
\caption{(a) and (b) both correspond to the structures induced by taking the equivalence class over paths between vertices.  In this illustration, all paths have the same source vertex and each path is labeled by a different color, so that the structure in (b) can be obtained by recoloring the paths from (a) (i.e., green to orange, red to green, and orange to red).}
\label{fig:path}
\end{figure}

\section{Relational Exchangeability}\label{section:preliminaries}

Let $(S,\mathcal{S})$ be any Borel space.  A countably infinite sequence of $S$-valued random variables $\Xbf=(X_1,X_2,\ldots)$ is {\em exchangeable} if $\Xbf^{\sigma}\equalinlaw\Xbf$ for every permutation $\sigma:\Nb\to\Nb$, where $\equalinlaw$ denotes {\em equality in distribution} and $\Xbf^{\sigma}:=(X_{\sigma(1)},X_{\sigma(2)},\ldots)$ is the reordering of $\Xbf$ according to $\sigma$.
By de Finetti's theorem \cite{deFinetti1937}, the distribution of every such countably infinite sequence can be expressed as a mixture of independent, identically distributed (i.i.d.)~sequences.
In particular, with $\mathcal{P}(S)$ denoting the space of probability measures on $(S,\mathcal{S})$, there is a unique probability measure $\phi$ on $\mathcal{P}(S)$ such that the distribution of $\Xbf$ can be expressed as
\begin{equation}\label{eq:definetti}
\mathbb{P}(\Xbf\in\cdot)=\int_{\mathcal{P}(S)}\nu^{\infty}(\cdot)\phi(d\nu),
\end{equation}
where $\nu^{\infty}$ denotes the infinite product measure induced by $\nu$ on $S^{\infty}$.  We call the measure $\phi$ in \eqref{eq:definetti} the {\em de Finetti measure} of $\Xbf$.

de Finetti's theorem figures into our treatment of relationally exchangeable structures, which we now define.  Note that all of the examples in Section \ref{section:introduction}, and many more that could arise in practice, involve relational structures as they are defined in the coming section.  Example \ref{ex:example} provides a concrete illustration.

\subsection{Relational structures}

Let $\ar:\Nb\to\Nb\cup\{0\}$ be a \emph{signature} 
such that $\ar(j)=0$ implies $\ar(k)=0$ for all $k\geq j$.
An {\em $\ar$-structure with domain $A$} is a collection $\mathcal{A}=(R_1^{\mathcal{A}}, R_2^{\mathcal{A}},\ldots)$, where $A$ is a set and each $R_j^{\mathcal{A}}\subseteq A^{\ar(j)}$ is a relation of arity $\ar(j)$ on $A$.
We adopt the convention $A^0:=\emptyset$ so that $R_j^{\mathcal{A}}=\emptyset$ whenever $\alpha(j)=0$.

We write $\dom\mathcal{A}:=A$ for the {\em domain} of $\mathcal{A}$
and $\fin_{\ar}$ for the set of all $\alpha$-structures 
$\mathcal{A}$ with finite
$\dom\mathcal{A}\subset\Zb=\{\ldots,-1,0,1,\ldots\}$ 
such that $\sum_{j=1}^{\infty}|R_j^{\mathcal{A}}|<\infty$, where $|S|$
denotes the cardinality of a set $S$.  Writing $\dom\mathcal{A}\finite\Zb$ to denote that $\dom\mathcal{A}$ is a finite subset of $\Zb$, we have
\[\fin_{\ar}:=\{\mathcal{A}=(R_1^{\mathcal{A}},R_2^{\mathcal{A}},\ldots): 
\dom\mathcal{A}\finite\Zb\text{ and }\sum_{j=1}^{\infty}|R_j^{\mathcal{A}}|<\infty\}.\]
Condition $\sum_{j\geq1}|R_j^{\mathcal{A}}|<\infty$ implies that to every $\mathcal{A}\in\fin_{\ar}$ there is an $r_{\mathcal{A}}:=\max\{k\geq1: R_k^{\mathcal{A}}\neq\emptyset\}<\infty$, making it convenient to sometimes express $\mathcal{A}\in\fin_{\ar}$ as $(R_1^{\mathcal{A}},\ldots,R_{r_{\mathcal{A}}}^{\mathcal{A}})$ 
by omitting the infinite sequence of empty relations after $r_{\mathcal{A}}$.

For any $A^*\supseteq A$ and any injection $\rho:A^*\to A'$ there is an induced action on $\mathcal{A}=(R_1^{\mathcal{A}},R_2^{\mathcal{A}},\ldots)$ by $\mathcal{A}\mapsto\rho(\mathcal{A})=(R_1^{\rho(\mathcal{A})}, R_2^{\rho(\mathcal{A})},\ldots)$ with $\dom\rho(\mathcal{A})=\rho(A):=\{\rho(a): a\in A\}$ and
\begin{equation}\label{eq:relabel}(\rho(a_1),\ldots,\rho(a_{\ar(j)}))\in R_j^{\rho(\mathcal{A})}\quad\text{if and only if}\quad(a_1,\ldots,a_{\ar(j)})\in R_j^{\mathcal{A}}\end{equation}
for each $j=1,2,\ldots$.

For $n\geq1$ and any at most countable subset of finite $\ar$-structures $\mathcal{R}\subseteq\fin_{\ar}$ for which each $\mathcal{A}\in\mathcal{R}$ has $\dom\mathcal{A}\subseteq\Zb_+:=\{1,2,\ldots\}$, we consider the set $\Rncong$ of equivalence classes of $\mathcal{R}$-valued sequences $\xbf:[n]\to\mathcal{R}$ obtained as follows.
Any $\xbf:[n]\to\mathcal{R}$ determines a {\em relationally labeled structure} by removing the labels of the elements contained in each of the $\xbf(i)$ while maintaining the integrity of the overall structure induced by $\xbf$.
Formally, we define the {\em relationally labeled structure induced by $\xbf$} as the equivalence class
\begin{equation}\label{eq:cong}\xbf_{\cong}:=\{\xbf':[n]\to\mathcal{R}: \rho \xbf'= \xbf\text{ for some bijection }\rho:\Zb\to\Zb\},\end{equation}
where $\rho \xbf:[n]\to\mathcal{R}$ is defined by $(\rho \xbf)(i)=\rho(\xbf(i))$ as in \eqref{eq:relabel} above.
We write $\Rncong$ as the set of all $\xbf_{\cong}$ constructed from
some $\xbf:[n]\to\mathcal{R}$ in this way.

\begin{rmk}[Notation]
Below we reserve lowercase letters $x,y,\ldots$ for generic (non-random) objects and uppercase letters $X,Y,\ldots$ for random objects.  We use bold letters $\xbf,\ybf,\ldots$ for generic (non-random) functions $\xbf:[n]\to\mathcal{R}$ and $\Xbf,\Ybf,\ldots$ for random functions $\Xbf:[n]\to\mathcal{R}$.
\end{rmk}

\begin{example}\label{ex:example}
In this example, we demonstrate how the structures in Section \ref{section:introduction} fit into the framework of $\alpha$-structures.  For this we write $\mathcal{A}=(R_1^{\mathcal{A}},R_2^{\mathcal{A}},\ldots)$ as an $\alpha$-structure with $\dom\mathcal{A}=A$ and $\alpha$ varying according to the context.  

For $\ar(1)=1$ and $\ar(k)=0$ for $k\geq2$, each $\ar$-structure 
$\mathcal{A}$ is determined by a subset $R_1^{\mathcal{A}}\subseteq A$.
If we then choose $\mathcal{R}\subseteq\fin_{\ar}$ to consist of all
singleton subsets of $\Nb$, so that each $R_1^{\mathcal{A}}$ has the
form $\{i\}$, $i\geq1$, then the equivalence class of
$\xbf:[n]\to\mathcal{R}$ in \eqref{eq:cong} corresponds to the 
equivalence relation induced by $\xbf$ as in \eqref{eq:equiv-relation}.  

For $\ar(1)=2$, $\ar(k)=0$ for $k\geq2$, each $\ar$-structure
$\mathcal{A}$ with $\dom\mathcal{A}=A$ corresponds to a binary relation $R_1^{\mathcal{A}}\subseteq A^2$.
Taking $\mathcal{R}\subseteq \fin_{\ar}$ to consist of all structures 
with $R_1^{\mathcal{A}}=\{(i,j)\}$ for $i\neq j$ corresponds to sampling
from the phone call database in Section \ref{section:introduction}: 
each sampled $X_k=\{(i,j)\}$ corresponds to a caller-receiver pair
$(i,j)$. The equivalence class $\xbf_{\cong}$ gives an edge-labeled
(directed) graph as in Figure \ref{fig:calls-exch}.  (To get an
undirected graph we would take each $R_1^{\mathcal{A}}$ to consist of symmetric pairs $\{(i,j),(j,i)\}$.)

Taking $\alpha(j) = j$ for all $j\geq1$ means that any
$\alpha$-structure is determined by a collection
$R_j^{\mathcal{A}}\subseteq A^j$ of subsets of $j$-tuples for every
$j\geq1$.  
If we take $\mathcal{R}\subset\fin_{\ar}$ to consist only
of those $\alpha$-structures of the form
$R_k^{\mathcal{A}}=\{(a_{\sigma(1)},a_{\sigma(2)},\ldots,a_{\sigma(k)}):
\text{permutations }\sigma:[k]\to[k] \}$ for some $k\geq1$ and
$R_j^{\mathcal{A}}=\emptyset$ for all $j\neq k$, then the elements
$\mathcal{A}\in\mathcal{R}$ can be used to represent the set of
coauthors in a sample of articles.

Also for $\alpha(j) = j$ for all $j \geq 1$, if we take
$\mathcal{R}\subset\fin_{\ar}$ to consist only of those
$\alpha$-structures of the form
$R_k^{\mathcal{A}}=\{(a_1,a_2,\ldots,a_{k})\}$ for some $k\geq1$ and
$R_j^{\mathcal{A}}=\emptyset$ for all $j\neq k$, then each
$\mathcal{A}\in\mathcal{R}$ corresponds to a path from $a_1$ to
$a_{k}$, as discussed at the end of Section
\ref{section:introduction}.

\end{example}

Our discussion below specializes to the case of random structures labeled by the countable set $\Nb$, 
which are those structures constructed just as in \eqref{eq:cong} but for a countable sequence $\xbf:\Nb\to\mathcal{R}$.  
(Note well the difference between the index set $\Nb$ and the domain of the structures in $\mathcal{R}$, 
which we take to be $\Zb_{+}$.  The domain labeling is ``quotiented out'' in \eqref{eq:cong}, while the indexing $\Nb$ remains, serving as the ``edge labels'' in the associated interpretation as an edge-labeled graph in Figure \ref{fig:calls-exch}.) 

We write $\RNcong$ to denote the set of such structures and we equip $\RNcong$ with the Borel $\sigma$-field associated to its product-discrete topology induced by the following metric.  For any $\xbf_{\cong}\in\RNcong$ and $n\geq1$, we define the restriction $\Rn\xbf_{\cong}$ of $\xbf_{\cong}$ to $\Rncong$ as the structure obtained by taking any $\xbf'\in\xbf_{\cong}$ (i.e., $\xbf':\Nb\to\mathcal{R}$ for which there exists $\rho$ such that $\rho\xbf'=\xbf$), restricting its domain to $[n]$ by $\xbf'_{|[n]}:[n]\to\mathcal{R}$, $i\mapsto \xbf'(i)$, and putting $\Rn\xbf_{\cong}:=(\xbf'_{|[n]})_{\cong}$ as in \eqref{eq:cong}.  By the definition of $\RNcong$, it is clear that this is well-defined and does not depend on the specific choice of $\xbf'\in\xbf_{\cong}$.  We then define the metric on $\RNcong$ by
\[d(\xbf_{\cong},\xbf'_{\cong}):=1/(1+\sup\{n\geq1: \Rn\xbf_{\cong}=\Rn\xbf'_{\cong}\}),\quad\xbf_{\cong},\xbf'_{\cong}\in\RNcong,\]
with the convention that $1/\infty=0$, under which $\RNcong$ is compact. 

For any permutation $\sigma:\Nb\to\Nb$, we define the {\em relabeling of $\xbf_{\cong}\in\RNcong$ by $\sigma$} as the structure $\xbf_{\cong}^{\sigma}$ obtained by first choosing any $\xbf'\in\xbf_{\cong}$ and putting $\xbf_{\cong}^{\sigma}:=(\xbf'\circ\sigma)_{\cong}$, where $\xbf'\circ\sigma:\Nb\to\mathcal{R}$ is defined by usual composition of functions, $(\xbf'\circ\sigma)(i):=\xbf'(\sigma(i))$.  It is once again clear that this does not depend on the specific choice of representative $\xbf'\in\xbf_{\cong}$ since the actions of $\rho:\Nb\to\Nb$ and $\sigma:\Nb\to\Nb$ commute for all $\xbf'$, i.e., $\rho(\xbf'\circ\sigma)=(\rho\xbf')\circ\sigma$.

\begin{defn}[Relational exchangeability]\label{defn:relational}
A random structure $\Xbf_{\cong}\in\RNcong$ is {\em relationally exchangeable} if $\Xbf^{\sigma}_{\cong}\equalinlaw \Xbf_{\cong}$ for all permutations $\sigma:\Nb\to\Nb$.
\end{defn}

\subsection{Representation theorem}

Our main theorem establishes a generic construction for infinite relationally exchangeable structures $\Xbf_{\cong}$ in $\RNcong$.  The construction proceeds by sampling a sequence $X_1,X_2,\ldots$ conditionally i.i.d.\ in a related space $\mathcal{R}^\star$ and then modifying these observations to obtain a new sequence $\Xbf^{\dagger}:=(X^\dagger_1,X^\dagger_2,\ldots)$.   We then construct the equivalence class $\Xbf^{\dagger}_{\cong}$ based on the sequence $X^\dagger_1,X^\dagger_2,\ldots$ as in \eqref{eq:cong}.  Since stating the theorem formally requires some new ideas and notation, we give the construction and key ideas of the proof prior to stating the result in Theorem \ref{thm:general}.

As above, let $\ar:\Nb\to\Nb\cup\{0\}$ and $\mathcal{R}\subseteq\fin_{\ar}$ be an at most countable set of $\ar$-structures such that $\ar(j)=0$ implies $\ar(k)=0$ for all $k\geq j$ and each $\mathcal{A}\in\mathcal{R}$ has $\dom\mathcal{A}\finite\Zb_+$.  Since $\mathcal{R}$ is at most countable, we may fix an ordering $\mathcal{R}=\{\mathcal{R}_n\}_{n\geq1}$ of its elements.
Given any $\mathcal{A}=(R_1^{\mathcal{A}},R_2^{\mathcal{A}},\ldots)\in\fin_{\ar}$ and any $[0,1]$-valued sequence $(T_i)_{i\in\Zb}$, we define $T\mathcal{A}:\equiv T\circ\mathcal{A}$ as the $\alpha$-structure $T\mathcal{A}$ with $\dom(T\mathcal{A}):=\{T_i: i\in\dom\mathcal{A}\}$ and relations $R_j^{T\mathcal{A}}$ given by
\begin{equation}\label{eq:xi-structure}
(T_{a_1},\ldots,T_{a_{\alpha(j)}})\in R_j^{T\mathcal{A}}\quad\text{if and only if}\quad(a_1,\ldots,a_{\alpha(j)})\in R_j^{\mathcal{A}}\end{equation}
for each $j=1,2,\ldots$.
We then define 
\[\mathcal{R}_{[0,1]}:=\bigcup_{n\geq1}\{T\mathcal{R}_n: T=(T_i)_{i\in\Zb}\in[0,1]^{\Zb}\}\]
 as the set of $\ar$-structures obtained by associating $[0,1]$-valued labels to the elements of $\mathcal{R}$. 
 More generally, if $(\Xi_{i})_{i\in\dom\mathcal{A}}$ is a collection of subsets $\Xi_i\subseteq[0,1]$, then we define $\Xi\mathcal{A}$ by
\[\Xi\mathcal{A}:=\{T\mathcal{A}: T_i\in\Xi_i\text{ for each }i\in\dom\mathcal{A}\}.\]
We equip $\mathcal{R}_{[0,1]}$ with the $\sigma$-field on $\mathcal{R}_{[0,1]}$ generated by all sets of the form $\Xi\mathcal{A}$ with $\mathcal{A}\in\mathcal{R}$ and $\Xi=(\Xi_i)_{i\in\dom\mathcal{A}}$ a collection of Borel subsets of $[0,1]$.

From any $\xbf_{\cong}\in\RNcong$ and any sequence $\xi=(\xi_i)_{i\in\Zb}$ of i.i.d.\ Uniform$[0,1]$ random variables, we write $\xi\xbf_{\cong}$ to denote a random $\Ro$-valued sequence $(Z_i)_{i\geq1}$ obtained by first taking any representative $\xbf'\in\xbf_{\cong}$ and then putting $Z_i=\xi\xbf'(i)$ for each $i\geq1$, with $\xi\xbf'(i)$ defined as in \eqref{eq:xi-structure}.
Since $\xbf_{\cong}$ is fixed in this example, each $\xi \xbf'(i)$ corresponds to an assignment of Uniform$[0,1]$ random labels to the elements in the domain of $\xbf'(i)$.  Since $\xi_i\neq\xi_j$ with probability 1 for all $i\neq j$, it is immediate that 
the distribution of $\xi\xbf_{\cong}$ does not depend on the manner in which the representative $\xbf'$ is chosen.  
Now for any infinite relationally exchangeable structure $\Xbf_{\cong}\in\RNcong$ and a sequence $\xi$ of Uniform$[0,1]$ random variables independent of $\Xbf_{\cong}$, $\xi \Xbf_{\cong}=(Z_i)_{i\geq1}$ defines a random sequence obtained by putting $Z_i=\xi\xbf'(i)$ for a representative $\xbf'\in\xbf_{\cong}$ on the event $\Xbf_{\cong}=\xbf_{\cong}$.
By exchangeability of $\Xbf_{\cong}$, the sequence $\Zbf=(Z_i)_{i\geq1}$ obtained in this way is also exchangeable and de Finetti's theorem implies that $\Zbf$ is distributed as an i.i.d.\ sequence from a random measure $\nu$ on $\Ro$, as described in \eqref{eq:definetti}.

Given $\nu$, we define the {\em propensity of $u\in[0,1]$ in $\Zbf$} by
\begin{equation}\label{eq:nu-star}\nu^\star(u):=\nu(\{\mathcal{A}\in\mathcal{R}_{[0,1]}: u\in\dom\mathcal{A}\}),\end{equation}
which equals the conditional probability of the event $\{u\in \dom Z_j\}$ given $\nu$, for each $j\geq1$.
It is clear that each $u\in[0,1]$ appears in either 0, 1, or infinitely many of the relations $(Z_i)_{i\geq1}$ with probability 1.  First, if $\nu^\star(u)>0$, then the strong law of large numbers implies that $u$ occurs in infinitely many of the relations $Z_i$ with proportion $\nu^\star(u)$.  If $\nu^\star(u)=0$ and $u\in Z_i$ for some $i$, then the probability that $u$ occurs in $Z_j$ is $\nu^\star(u)=0$ independently for each $j\neq i$ and therefore $u$ appears only once in $\Zbf$ with probability 1.

Clearly the set $\mathcal{U}=\{u:\nu^\star(u)>0\}$ is at most countable since $|\dom(Z_i)|<\infty$ for each $i=1,2,\ldots$. We can, therefore, order the elements of $\mathcal{U}$ as $u_1,u_2,\ldots$ such that $\nu^\star(u_j)\geq\nu^\star(u_{j+1})$ for $j\geq1$, breaking ties $\nu^\star(v)=\nu^\star(w)$ as follows.  For each $\mathcal{A}\in\mathcal{R}$, we define
\[\nu^\star(u;\mathcal{A}):=\nu(\{T\mathcal{A}: T=(T_i)_{i\in\Zb}\in[0,1]^{\Zb}\text{ such that } u\in\dom(T\mathcal{A})\}),\quad u\in[0,1],\]
to be the measure assigned to the subset of $\mathcal{R}_{[0,1]}$ whose structure is consistent with $\mathcal{A}$ and which contains $u$ in its domain.\label{tiebreak}
Assuming $\nu^\star(v)=\nu^\star(w)$, we assign the smaller label to $v$ if there is some $k\geq1$ such that $\nu^\star(v;\mathcal{R}_j)=\nu^\star(w;\mathcal{R}_j)$ for all $j<k$ and $\nu^\star(v;\mathcal{R}_k)>\nu^\star(w;\mathcal{R}_k)$.
If $\nu^\star(v;\mathcal{R}_j)=\nu^\star(w;\mathcal{R}_j)$ for all $j\geq1$, then we label $v$ and $w$ in increasing order, i.e., if we are to assign labels $j$ and $j+1$ to $v$ and $w$ and if $v<w$, then we shall put $u_j=v$ and $u_{j+1}=w$; otherwise, we put $u_j=w$ and $u_{j+1}=v$.
The ordering $\mathcal{U}=(u_j)_{j\geq1}$ is thus uniquely determined by $\nu$ and the fixed ordering of $\mathcal{R}$ chosen at the outset.

Now given $\mathcal{U}$, for any $\mathcal{A}\in\mathcal{R}_{[0,1]}$ we define $\mathcal{A}^\star$ (suppressing the dependence on $\nu$) by replacing each occurrence of $u_j\in\mathcal{U}$ by $j$ and replacing each occurrence of $v'\notin\mathcal{U}$ in $\mathcal{A}$ by a unique non-positive integer $z(v')=0,-1,-2,\ldots$ so that for $v',v''\notin\mathcal{U}$ and both in $\dom\mathcal{A}$ with $v'\leq v''$ we have $z(v')\leq z(v'')$ and the $z(v')$ are chosen to be the largest possible non-positive integers that satisfy this condition.  (For example, if $v_1<v_2<v_3$ are the only elements in $\dom\mathcal{A}$ that are not in $\mathcal{U}$, we assign $z(v_3)=0$, $z(v_2)=-1$, and $z(v_1)=-2$.)

 Every $\mathcal{A}\in\mathcal{R}_{[0,1]}$ thus corresponds to a unique such $\mathcal{A}^\star$ and we define $\mathcal{R}^\star$ as the set of all structures $\mathcal{A}^\star$ obtained in this way.  Note that since $\mathcal{R}$ is at most countable, so is $\mathcal{R}^\star$.  Also, although we have constructed $\mathcal{R}^\star$ using $\mathcal{U}$ (and therefore $\nu$), the set $\mathcal{R}^\star$ depends only on $\mathcal{R}$, justifying the term {\em $\mathcal{R}$-simplex} in our definition of $\FR$ in \eqref{eq:FR} below.

The elements of $\mathcal{R}^\star$ are thus $\ar$-structures $\mathcal{A}^\star$ with $\dom\mathcal{A}^\star\subseteq\Zb$.
We define the $\mathcal{R}$-simplex $\FR$ by 
\begin{equation}\label{eq:FR}\FR:=\left\{(f_{\mathcal{B}})_{\mathcal{B}\in\mathcal{R}^\star}:\ f_{\mathcal{B}}\geq0\text{ and } \sum_{\mathcal{B}\in\mathcal{R}^\star}f_{\mathcal{B}}=1\right\},\end{equation}
on which we equip the metric
\[d_{\FR}(f,f'):=\sum_{\mathcal{B}\in\mathcal{R}^\star}|f_{\mathcal{B}}-f'_{\mathcal{B}}|,\quad f,f'\in\FR,\]
and the associated Borel $\sigma$-field.

Any $f\in\FR$ determines a unique probability measure $\varepsilon_f$ on $\RNcong$ by first drawing $\Xbf=(X_1,X_2,\ldots)$ i.i.d.\ from
\begin{equation}\label{eq:f-induced}\mathbb{P}(X_i=\mathcal{B}\mid f)=f_{\mathcal{B}},\quad\mathcal{B}\in\mathcal{R}^\star,\end{equation}
and then constructing $\Xbf^\dagger=(X^\dagger_1,X^\dagger_2,\ldots)$ from $\Xbf$ as follows.
We initialize by putting $m_0=0$.
For each $n\geq1$, given $m_{n-1}$, we replace the non-positive elements in $\dom X_n$ according to the rule:
\begin{itemize}
	\item[(i)]  If $\dom X_n$ has no non-positive elements, then put $m_n=m_{n-1}$ and $X^\dagger_n:=X_n$.  
	\item[(ii)] If $\dom X_n$ has non-positive elements $0,\ldots,-k$ for $k\geq0$, then define $X^\dagger_n$ by replacing each occurrence of $-i$, for $0\leq i\leq k$, in $X_n$ by $m_{n-1}-i$, then putting $m_{n}=m_{n-1}-k-1$ and keeping positive elements unchanged.  See Example \ref{example:partition} below for an illustration.
\end{itemize}
Note the distinction between the use of non-positive elements in construction of $X_i$ and $X_i^\dagger$ respectively.
The non-positive elements of each $X_i \in \mathcal{R}^\star$ serve to denote the non-recurring particles that appear \emph{only within} this particular relation.
The non-positive elements of $X_i^\dagger$, on the other hand, serve to denote the non-recurring particles across \emph{all} relations in the
sequence~$\Xbf^\dagger$.
We define $\varepsilon_f$ to be the distribution of $\Xbf^\dagger_{\cong}$ constructed by applying \eqref{eq:cong} to the sequence $\Xbf^\dagger$.

From $\nu$, we define $f^{\nu}=(f_{\mathcal{B}}^{\nu})_{\mathcal{B}\in\mathcal{R}^\star}\in\FR$ by
\begin{equation}\label{eq:canonical}f^{\nu}_{\mathcal{B}}:=\nu(\{\mathcal{A}\in\mathcal{R}_{[0,1]}: \mathcal{A}^\star=\mathcal{B}\}),\quad\mathcal{B}\in\mathcal{R}^\star.\end{equation}
This choice of $f^{\nu}=(f^{\nu}_{\mathcal{B}})_{\mathcal{B}\in\mathcal{R}^\star}$ is uniquely determined by $\nu$ and the fixed ordering of $\mathcal{R}$.
Conversely, given $f=(f_{\mathcal{B}})_{\mathcal{B}\in\mathcal{R}^\star}\in\FR$ we construct a measure $\nu_f$ on $\mathcal{R}_{[0,1]}$ to be the distribution of $\xi Y$ from \eqref{eq:xi-structure} for $Y$ drawn from distribution \eqref{eq:f-induced} and $\xi=(\xi_{i})_{i\in\mathbb{Z}}$ i.i.d.\ Uniform$[0,1]$ independent of $Y$.
Proposition~\ref{prop:nu} proves that the above procedure
does not alter the random, relationally labeled structure.

\begin{prop}\label{prop:nu}
Let $\Xbf_{\cong}$ be relationally exchangeable and let $\theta=(\theta_i)_{i\in\Zb}$ be i.i.d.\ Uniform$[0,1]$ independent of $\Xbf_{\cong}$.  Then $(((\theta\Xbf_{\cong})^{\star})^{\dagger})_{\cong}=\Xbf_{\cong}$ a.s., where $(\theta\Xbf_{\cong})^\star$ denotes the application of $^{\star}:\mathcal{R}_{[0,1]}\to\mathcal{R}^\star$ to each component of the sequence $\theta\Xbf_{\cong}$.
\end{prop}

\begin{proof}
Let $\theta=(\theta_i)_{i\in\Zb}$ be i.i.d.\ Uniform$[0,1]$ independently of $\Xbf_{\cong}$.  
Each event $\Xbf_{\cong}=\xbf_{\cong}$ gives rise to a probability measure $\nu$ on $\mathcal{R}_{[0,1]}$ through the de Finetti measure of $\theta\xbf_{\cong}$; see \eqref{eq:definetti}.  Let $\mathcal{U}=(u_i)_{i\geq1}$ be the ordered subset of $[0,1]$ corresponding to the atoms of $\nu^\star$ as in \eqref{eq:nu-star}.  Since $\theta$ are i.i.d.\ Uniform$[0,1]$, we have
\[\mathbb{P}\{\theta_i\neq\theta_j\text{ for all }i\neq j\}=1,\]
which implies that distinct $i,j\in\Zb$ are labeled distinctly in $\theta\xbf_{\cong}$ with probability 1.
In particular, the non-positive labels in $(\theta\xbf_{\cong})^\star$ account for all those elements that appear in only one entry of the sequence $\theta\xbf_{\cong}$.
It follows that $((\theta\xbf_{\cong})^\star)^{\dagger}\in\xbf_{\cong}$ with probability 1 and, thus, $(((\theta\xbf_{\cong})^\star)^{\dagger})_{\cong}=\xbf_{\cong}$ with probability 1 for all possible outcomes $\Xbf_{\cong}=\xbf_{\cong}$.
\end{proof}

\begin{example}\label{example:partition}
To illustrate the above procedure, let $\mathcal{R}$ consist of the singleton sets $(\{i\})$, $i\geq1$, just as in the first part of Example \ref{ex:example}, written as $\{1\},\{2\},\ldots$ for simplicity.
Let $\mathcal{U}=(u_i)_{i\geq1}$ be a countable subset of $[0,1]$.
From $\mathcal{R}=\{\{1\},\{2\},\ldots\}$, we obtain $\mathcal{R}^\star=\{\{0\},\{1\},\ldots\}$ since for any sequence $T=(T_i)_{i\in\Zb}$ and any $\{i\}$, $i\geq1$, the transformed relation $T\{i\}\equiv\{T_i\}$ has either $T_i\in\mathcal{U}$ or $T_i\not\in\mathcal{U}$.  If $T_i=u_k\in\mathcal{U}$, then $\{T_i\}^\star=\{k\}$ for $k\geq1$; and if $T_i=v\not\in\mathcal{U}$, then $\{T_i\}^\star = (\{0\})$.

Given $(f_{\{0\}},f_{\{1\}},\ldots)\in\FR$, we generate a sequence $X_1,X_2,\ldots$ of singleton sets i.i.d.\ as in \eqref{eq:f-induced}, from which we obtain $X_1^{\dagger},X_2^{\dagger},\ldots$ by reassigning occurrences of $0$ to the greatest negative integer that has not yet appeared in the sequence.
For example, if the sequence begins $\{4\},\{0\},\{2\},\{0\},\{2\},\{2\},\{0\},\ldots$, then we reassign labels to obtain $\{4\},\{0\},\{2\},\{-1\},\{2\},\{2\},\{-2\},\ldots$.
Delabeling according to \eqref{eq:cong} (or equivalently
\eqref{eq:equiv-relation}) yields the equivalence classes $\{1\}$,
$\{2\}$, $\{3,5,6\}$, $\{4\}$, $\{7\}$.

For another example, let $\mathcal{R}$ consist of ordered pairs $\{(i,j)\}$, $1\leq i\neq j<\infty$, so that we are in the phone call example of Section \ref{section:introduction}.  Let $\mathcal{U}=(u_i)_{i\geq1}$ be a countable subset of $[0,1]$.
Then $\mathcal{R}=\{\{(i,j)\}: i\neq j\geq1\}$ and 
\[\mathcal{R}^{\star}=\{\{(i,0)\}: i\geq1\}\cup\{\{(0,i)\}: i\geq1\}\cup\{\{(0,-1)\},\{(-1,0)\}\}\cup\mathcal{R}\] since any $T=(T_i)_{i\in\Zb}$ and any $\{(i,j)\}\in\mathcal{R}$ is transformed by $T\{(i,j)\}=\{(T_i,T_j)\}$, which may have $T_i=u_{i'}\in\mathcal{U}$ and $T_j=u_{j'}\in\mathcal{U}$ in which case $\{(T_i,T_j)\}^{\star}=\{(i',j')\}$.  If $T_i=u_{i'}\in\mathcal{U}$ and $T_j\not\in\mathcal{U}$, then $\{(T_i,T_j)\}^{\star}=\{(i',0)\}$.  If $T_i\not\in\mathcal{U}$ and $T_j=u_{j'}$ then $\{(T_i,T_j)\}^{\star}=\{(0,j')\}$.  If $T_i,T_j\not\in\mathcal{U}$ then $\{(T_i,T_j)\}^{\star}$ will equal $\{(0,-1)\}$ or $\{(-1,0)\}$ depending on whether $T_j<T_i$ or $T_i<T_j$, respectively.
\end{example}

The above construction gives the following representation theorem.

\begin{thm}\label{thm:general}
Let $\ar$ be a signature, $\mathcal{R}\subseteq\fin_{\ar}$, and fix an ordering $\mathcal{R}=(\mathcal{R}_n)_{n\geq1}$.
Let $\Xbf_{\cong}$ be a relationally exchangeable random structure in $\RNcong$.
Then there exists a probability measure $\phi$ on $\FR$ such that $\Xbf_{\cong}\sim\varepsilon_{\phi}$, where
\begin{equation}\label{eq:epsilon-phi}\varepsilon_{\phi}(\cdot)=\int_{\FR}\varepsilon_f(\cdot)\phi(df).\end{equation}

\end{thm}

A canonical version of the measure $\phi$ in Theorem \ref{thm:general}
can be constructed as in \eqref{eq:canonical} above, and this is the measure we construct in the following proof.

\begin{proof}
We proceed by constructing $\phi$ as in \eqref{eq:canonical} and showing that $\Ybf_{\cong}^{\dagger}\sim\varepsilon_{\phi}$ satisfies $\Ybf_{\cong}^{\dagger}\equalinlaw\Xbf_{\cong}$.

To see this, we first let $\psi$ be the de Finetti measure on $\mathcal{P}(\mathcal{R}_{[0,1]})$ (i.e., the space of probability measures on $\mathcal{R}_{[0,1]}$) associated to $\xi\Xbf_{\cong}$ for $\xi=(\xi_{i})_{i\in\Zb}$ i.i.d.\ Uniform$[0,1]$ independently of $\Xbf_{\cong}$.  In particular, the distribution of $\xi\Xbf_{\cong}$ is conditionally i.i.d.\ from $\nu\sim\psi$.  
 The measure $\psi$ determines a measure $\phi$ on $\FR$ through \eqref{eq:canonical}.  
 
Given $f\sim\phi$, we construct $\Ybf=(Y_1,Y_2,\ldots)$ as conditionally i.i.d.\ from the distribution in \eqref{eq:f-induced} and $\Ybf^\dagger$ as in (i) and (ii) above.  Let $\mathcal{U}=(u_i)_{i\geq1}$ be the atoms of $\theta\Ybf^{\dagger}$ for $\theta=(\theta_i)_{i\in\Zb}$ i.i.d.\ Uniform$[0,1]$ independent of $\Ybf^{\dagger}$, let $\zeta=(\zeta_{k})_{k\leq0}$ be i.i.d.\ Uniform$[0,1]$ independently of everything else, and define $\xi_f=(\xi_i)_{i\in\mathbb{Z}}$ by
\[\xi_i=\left\{\begin{array}{cc} u_i,& i\geq1,\\ \zeta_i,&
                                                           \text{otherwise.}\end{array}\right.\]
First note that the conditional distribution of $\xi_f\Ybf^\dagger$
given $f$ is the same as the conditional law of ${\theta}\Ybf^{\dagger}$
given $\nu_f$, since the $\xi_i$ for $i\geq1$ were constructed from the atoms of $\theta\Ybf^{\dagger}_{\cong}$.    Writing $\mathcal{D}(X\mid Y)$ to denote the conditional distribution of $X$ given $Y$, we thus have
\begin{align}
\mathcal{D}(\xi_f\Ybf_{\cong}^{\dagger}\mid f)&=\mathcal{D}(\theta\Ybf^{\dagger}\mid\nu_f)\label{eq:1}\quad\text{and}\\
\mathcal{D}(\xi_f\Ybf_{\cong}^{\dagger}\mid f)&=\mathcal{D}(\theta\Xbf_{\cong}\mid\nu_f),\label{eq:2}
\end{align}
where \eqref{eq:2} follows from the construction of $\phi$ from the de Finetti measure $\psi$ of $\xi\Xbf_{\cong}$.
From \eqref{eq:1} and \eqref{eq:2} it follows that
\begin{align*}
\mathcal{D}((((\xi_f\Ybf_{\cong}^{\dagger})^{\star})^{\dagger})_{\cong}\mid f)&=\mathcal{D}((((\theta\Ybf^{\dagger})^{\star})^{\dagger})_{\cong}\mid\nu_f)\quad\text{and}\\
\mathcal{D}((((\xi_f\Ybf^{\dagger}_{\cong})^\star)^{\dagger})_{\cong}\mid f)&=\mathcal{D}(((({\theta}
\Xbf_{\cong})^\star)^{\dagger})_{\cong}\mid \nu_f), 
\end{align*}
and thus
\[\mathcal{D}((((\theta\Ybf^{\dagger})^{\star})^{\dagger})_{\cong}\mid\nu_f)=\mathcal{D}(((({\theta}
\Xbf_{\cong})^\star)^{\dagger})_{\cong}\mid \nu_f).\]
By Proposition \ref{prop:nu}, we have
$(((\theta\Ybf^{\dagger}_{\cong})^{\star})^\dagger)_{\cong}=\Ybf^{\dagger}_{\cong}$
a.s.\ and
$((({\theta}\Xbf_{\cong})^\star)^{\dagger})_{\cong}=\Xbf_{\cong}$,
implying $\Ybf^{\dagger}_{\cong}\equalinlaw\Xbf_{\cong}$ as desired.

\end{proof}

\subsection{Special cases}

As discussed in Example \ref{example:partition}, the case in which $\mathcal{R}$ corresponds to the singleton sets $\{i\}$ for $i\geq1$ gives $\FR$ equal to the ranked simplex
\[\Delta^{\downarrow}:=\left\{(f_0,f_1,\ldots): f_1\geq f_2\geq\cdots\geq0,\ f_0\geq0,\ \sum_{i\geq0}f_i=1\right\}.\]
To see this note that $\mathcal{R}_{[0,1]}$ is the set of singleton elements of $[0,1]$, i.e.,
\[\mathcal{R}_{[0,1]}:=\{\{u\}: u\in[0,1]\}.\]
An exchangeable sequence $\Xbf=(X_1,X_2,\ldots)$ in $\mathcal{R}_{[0,1]}$ gives rise to a random countable subset $\mathcal{U}\subset[0,1]$ of elements that appear infinitely often among the $X_i$ and its complement $[0,1]\setminus\mathcal{U}$ consisting of all elements appearing at most once among the $X_i$.  In the construction of $\Xbf^\dagger$ from $\Xbf$ described above, any occurrence $X_i=\{u\}$ for $u\notin\mathcal{U}$ gives rise to $X_i^{\star}=\{0\}$, explaining why the definition of $\FR$ in \eqref{eq:FR} corresponds to the simplex of elements $(f_i)_{i\geq0}$.

Following Example \ref{ex:example}, the case of edge exchangeable random (directed) graphs corresponds to $\mathcal{R}\subset\fin_{\ar}$ with $\ar(1)=2$ and $\ar(k)=0$ for $k\geq2$, with each $\mathcal{A}\in\mathcal{R}$ having $R_1^{\mathcal{A}}=\{(i,j)\}$, $i\neq j$.  In this case, $\mathcal{R}_{[0,1]}$ corresponds to all pairs $(u,v)$ for $u,v\in[0,1]$.  An exchangeable sequence $\Xbf$ in $\mathcal{R}_{[0,1]}$ determines a random subset $\mathcal{U}\subset[0,1]$.  Now, in each $X_i$ there can be $0$, $1$, or $2$ elements $v\notin\mathcal{U}$.  Occurrences of such elements are replaced by $0$ and $-1$, as explained in Example \ref{example:partition}.

In general, we call the occurrences of any $u\notin\mathcal{U}$ in the sequence $\Xbf$ {\em blips}.  These elements are merely a `blip' in the overall sequence $\Xbf$---they occur only for an instant and then never again.  Our labeling convention in defining $\mathcal{R}^\star$ is that the non-positive labels correspond to the blips.  Theorem \ref{thm:general} describes how blips arise in general relationally exchangeable structures.

\section{Concluding}

Recent work in the statistical analysis of network data underscores the significance of relationally labeled structures in applications, as many data structures which are typically represented graphically, such as social networks and networks detailing email correspondence and professional collaborations, arise from a process by which interactions or relations accumulate within a population of otherwise indistinguishable individuals.
The work in \cite{CD2016e2} focused on the case of edge and hyperedge exchangeable random graphs, but the additional examples in Section \ref{section:introduction} involving repeated path sampling and snowball sampling are also highly relevant in networks applications.
The representation theorem serves two immediate statistical purposes.
First, the representation characterizes a general class of nonparametric statistical models of potential interest in the aforementioned applications.
Second, it establishes that vertices arrive in size-biased random order in relationally exchangeable structures, explaining why the common assumption of exchangeable vertex labeling, as presented in graphon models \cite{LovaszSzegedy2006}, is not tenable in many applications.
We reserve discussion of these practical implications for other work; see \cite{CD2016e2,CD2015AOS}.

\end{document}